\documentclass{article}
\usepackage[utf8]{inputenc}
\usepackage[T1]{fontenc}
\usepackage[english]{babel}
\usepackage{amsmath,amsthm,amssymb,amsfonts}
\usepackage{xspace}
\usepackage{verbatim}
\usepackage[hmargin=2.5cm,vmargin=3.8cm]{geometry}
\usepackage[textsize=footnotesize,color=green!40,textwidth=22mm]{todonotes}
\usepackage{hyperref}
\usepackage{tikz}
\usetikzlibrary{positioning}
\usetikzlibrary{automata, arrows}
\usetikzlibrary{snakes}
\usetikzlibrary{arrows.meta}
\tikzstyle{small node}=[circle, fill, draw, inner sep=0pt, minimum width=6pt]

\newtheorem{theorem}{Theorem}
\newtheorem*{theorem*}{Theorem}
\newtheorem{lemma}[theorem]{Lemma}
\newtheorem{proposition}[theorem]{Proposition}
\newtheorem{corollary}[theorem]{Corollary}

\newtheorem{definition}[theorem]{Definition}

\title{Extremal digraphs for open neighbourhood location-domination and identifying codes\footnote{The first author's research was financed by the French government IDEX-ISITE initiative 16-IDEX-0001 (CAP 20-25) and by the ANR project GRALMECO (ANR-21-CE48-0004). The research of the second author was in part supported by a grant from IPM (No. 1402050116).}}
\author{Florent Foucaud\footnote{\noindent Universit\'e Clermont Auvergne, CNRS, Clermont Auvergne INP, Mines Saint-\'Etienne, LIMOS, 63000 Clermont-Ferrand, France.} \and Narges Ghareghani\footnote{Department of Industrial Design, College of  Fine Arts, University of Tehran, Tehran, Iran.}
	\footnote{\noindent School of Mathematics, Institute for Research in Fundamental Sciences (IPM), Tehran, Iran.}\and Pouyeh Sharifani}

\begin{document}
	\maketitle
	
	\begin{abstract}
		A set $S$ of vertices of a digraph $D$ is called an open neighbourhood locating-dominating set if every vertex in $D$ has an in-neighbour in $S$, and for every pair $u,v$ of vertices of $D$, there is a vertex in $S$ that is an in-neighbour of exactly one of $u$ and $v$. The smallest size of an open neighbourhood locating-dominating set of a digraph $D$ is denoted by $\gamma_{OL}(D)$. We study the class of digraphs $D$ whose only open neighbourhood locating-dominating set consists of the whole set of vertices, in other words, $\gamma_{OL}(D)$ is equal to the order of $D$. We call those digraphs \emph{extremal}. By considering digraphs with loops allowed, our definition also applies to the related (and more widely studied) concept of identifying codes. We extend previous studies from the literature for both open neighbourhood locating-dominating sets and identifying codes of both undirected and directed graphs. These results all correspond to studying open neighbourhood locating-dominating sets on special classes of digraphs. To do so, we prove general structural properties of extremal digraphs, and we describe how they can all be constructed. We then use these properties to give new proofs of several known results from the literature. We also give a recursive and constructive characterization of the extremal di-trees (digraphs whose underlying undirected graph is a tree).
	\end{abstract}
	
	\section{Introduction}
	
	We consider extremal questions regarding the open neighbourhood location-domination problem on directed graphs (digraphs for short). This problem is part of the area of identification problems in discrete structures (such as graphs, digraphs or hypergraphs). In this type of problems, one wishes to uniquely determine some elements of the structure (usually the vertices or the edges) by means of a solution set (of vertices, edges or substructures), in the sense that each element to be distinguished is covered by a unique subset of the solution. Problems of this kind have been studied under various names and in different contexts such as \emph{separating systems}, \emph{discriminating codes}, or \emph{test collections}, see for example~\cite{BS07,B72,CCCHL08,HY14,MS85,R61}. They have many applications to various domains such as biological testing~\cite{MS85}, threat detection in facilities~\cite{UTS04} or fault diagnosis in computer networks~\cite{KCL98,R93}.
	
	\paragraph{Definitions.} In this paper, we consider directed graphs (digraphs for short) which can contain loops (a \emph{loop} is an arc from a vertex to itself). The vertex set and arc set of a digraph $D$ is denoted by $V(D)$ and $A(D)$, respectively. An arc from vertex $x$ to vertex $y$ is denoted $xy$, its \emph{tail} is $x$ and its \emph{head} is $y$. Multiple arcs between the same pair of vertices are allowed, but two arcs with the same tail and head are meaningless. Hence, we assume there are no multiple arcs. A digraph with no loops and with at most one arc between any pair of vertices is called an \emph{oriented graph}. A digraph is called \emph{reflexive} if each vertex has a loop. The in-neighbourhood of a vertex $x$ of $D$ is denoted by $N_D^-(v)$, and similarly $N_D^+(v)$ is the out-neighbourhood of $v$ (we may drop the $D$ subscripts if $D$ is clear from the context). A \emph{source} is a vertex with no in-neighbour, and a \emph{sink} is a vertex with no out-neighbour. By the \emph{underlying graph} of a digraph $D$, we mean the undirected simple graph (without loops and repeated edges) on vertex set $V(D)$ obtained from $D$ by adding an edge between $x$ and $y$ if $x\neq y$ and there exists an arc in $D$ between $x$ and $y$. A \emph{di-tree} is a digraph whose underlying graph is a tree. A \emph{rooted directed tree} is a directed graph without loops and directed 2-cycles whose underlying graph is a tree, which contains a single source called \emph{root}, and where each arc is oriented away from the root.
	
	We say that a digraph is \emph{connected} if its underlying graph is connected (this corresponds to the notion of weak connectivity of digraphs). If a digraph is not connected, we refer to its \emph{connected components} as the digraphs formed by the connected components of its underlying graph. A \emph{directed cycle} is a sequence of arcs such that the head of each arc is the same as the tail of the next one, the head of the last arc is the same as the tail of the first arc, and every vertex occurs only in two arcs of the sequence.	
	
	\paragraph{OLD sets.} The concept of open neighbourhood locating-dominating sets (OLD sets for short) was defined for undirected graphs under the name of \emph{IDNT codes} by Honkala, Laihonen and Ranto in~\cite[Section 5]{IDNT} and independently rediscovered by Seo and Slater in~\cite{SS10,SS11}, who coined the term ``OLD set''. We extend the definition to digraphs, in the same way as the definition of dominating sets of undirected graphs is classically extended to digraphs~\cite{Fu}. Given a digraph $D$, a set $S$ of vertices is an \emph{open neighbourhood locating-dominating set} of $D$ if (i) every vertex has an in-neighbour in $S$ (open neighbourhood domination condition) and (ii) for every pair of vertices, there is a vertex of $S$ that is an in-neighbour of exactly one of the two vertices (open neighbourhood location condition). The \emph{open neighbourhood location-domination number} (\emph{OLD number} for short) of $D$, denoted $\gamma_{OL}(D)$, is the smallest size of an OLD set of $D$. Note that a digraph with a vertex of in-degree~0 or with two vertices with the same in-neighbourhood (called \emph{in-twins}), does not admit any OLD set, but if the graph does not contain any such vertices, the whole vertex set is an OLD set. A digraph is called \emph{locatable} if it admits an OLD set.
	
	Since their introduction over a decade ago, OLD sets have been extensively studied, see~\cite{Chellali,interval-bounds,interval-algo,circulant,HY14,triangular,PP,SS11} for some papers on the topic. The concept of OLD sets is related to the one of \emph{locating-dominating sets}, defined by Slater in the 1980s~\cite{S87,S88}, where the open neighbourhood domination condition is replaced by closed neighbourhood domination, and the location condition is only required for pairs of vertices that are not in the solution set. In the related notion of \emph{identifying codes}, one replaces open (in-)neighbourhoods in both conditions by closed (in-)neighbourhoods. More precisely, a set $S$ of vertices is an identifying code of a digraph $D$ if (i) every vertex of $D$ has a vertex of $S$ in its closed in-neighbourhood and (ii) for every pair of vertices, there is a vertex of $S$ that belongs to the closed in-neighbourhood of exactly one of the two vertices. These notions were mainly studied for undirected graphs, but locating-dominating sets of digraphs were studied in~\cite{BBLP21,BDLP21,CHL02,FHP20,S07} and identifying codes of digraphs were studied
	in~\cite{BWLT06,CGHLMM06,CHL02,CH18,FNP13,S07}.
	
	In this paper, our goal is to study those locatable digraphs whose only OLD set is the whole set of vertices, which we call \emph{extremal digraphs}.

	\paragraph{Previous results.} All undirected graphs whose only OLD set is the whole set of vertices were characterized in~\cite{FGRS21}, as the family of \emph{half-graphs} defined in~\cite{EH84} (a half-graph is a special bipartite graph with both parts of the same size, where each part can be ordered so that the open neighbourhoods of consecutive vertices differ by exactly one vertex). Digraphs with no directed 2-cycles whose only identifying code is the whole set of vertices were characterized in~\cite{FNP13}, as transitive closures of top-down oriented forests. The aim of this paper is to study the set of digraphs whose only OLD set is the whole set of vertices. In fact, the above results can be reformulated in our setting.
	
	When a vertex has a loop, then its open in-neighbourhood is the same as its closed in-neighbourhood. Thus, for a reflexive digraph, the concept of an OLD set is the same as the one of an identifying code, as defined above. A digraph is \emph{symmetric} if for each arc $xy$, the arc $yx$ also exists. A symmetric digraph can be seen as an undirected graph. Thus, considering OLD sets of digraphs where loops are allowed, generalizes previous works on identifying codes of both digraphs and undirected graphs, and on OLD sets of undirected graphs.
	
	Using the digraph terminology, we can reformulate existing results from the literature in our setting as follows (the two first theorems were proved in the context of identifying codes). 
	
	\begin{theorem}[\cite{GM07}]\label{thm:IDcodes-n-1}
		For a connected, symmetric and reflexive locatable digraph $D$ of order $n$, $\gamma_{OL}(D)=n$ if and only if $n=1$.
	\end{theorem}
	
	\begin{theorem}[{\cite[Theorem 9]{FNP13}}]\label{thm:IDcodes-oriented}
		For a connected and reflexive locatable digraph $D$ of order $n$ without directed 2-cycles, $\gamma_{OL}(D)=n$ if and only if the digraph obtained from $D$ by removing all loops is the transitive closure of a rooted directed tree.
	\end{theorem}
	
	\begin{theorem}[{\cite[Theorem 1]{FGRS21}}]\label{thm:half-graphs}
		For a connected, symmetric and loop-free locatable digraph $D$ of order $n$, $\gamma_{OL}(D)=n$ if and only if the underlying graph of $D$ is a half-graph.
	\end{theorem}

	Also note that both OLD sets and identifying codes can be seen as a special case of discriminating codes in bipartite graphs, studied in~\cite{CCCHL08,CCHL08}. Given a bipartite graph $G$ with partite sets $I$ and $A$, a \emph{discriminating code} of $G$ is a subset $C$ of vertices of $A$ such that each vertex of $I$ has a unique and nonempty neighbourhood within $C$. Given a digraph $D$, one can construct a bipartite graph where $I$ and $A$ are two copies of $V(D)$ and a vertex in $I$ is adjacent to all vertices in $A$ corresponding to its in-neighbours in $D$. Now, a subset $C$ of $A$ is a discriminating code in the bipartite graph if and only if it is an OLD set in $D$. A similar construction (with closed in-neighbourhoods instead of open in-neighbourhoods) can be done for identifying codes~\cite{FNP13}. The problem of studying those bipartite graphs where all vertices of $A$ are required in any discriminating code was one of the main problems studied in~\cite{CCCHL08}, and thus the present paper partially answers this question.
	
	\paragraph{Our results.} We first study, in Section~\ref{sec:general}, general properties of digraphs of order $n$ with OLD number $n$. In such digraphs, every vertex is needed in every OLD set, either to dominate a vertex, or to locate a pair of vertices. Such vertices are called \emph{forced}. We show that however, in such a digraph, no vertex can be double-forced (i.e. forced because of two distinct reasons). We also show that the vertex set of such graphs can always be partitioned into subsets, each of which contains a spanning directed cycle.
	We then give a characterization of the (very rich) class of digraphs of order $n$ with OLD number $n$. We use the found structural properties and the characterization to give new proofs of Theorem~\ref{thm:IDcodes-n-1}, Theorem~\ref{thm:IDcodes-oriented} and Theorem~\ref{thm:half-graphs} in Section~\ref{sec:known}. Then, we focus, in Section~\ref{sec:trees}, on the class of extremal di-trees, 
	and give a recursive and constructive characterization of these digraphs. We conclude in Section~\ref{conclu}.

	\section{Structural properties of extremal digraphs}
\label{sec:general}

We now describe the structure of digraphs whose only OLD set is the whole vertex set. There are many such digraphs, as we will see. To achieve this, we will first prove some preliminary results.

\subsection{Forced vertices} 

In a locatable digraph $D$, some vertices have to belong to any OLD set: we call such vertices \emph{forced}, as was done in e.g.~\cite{FP12} in the context of identifying codes. There are two types of forced vertices: those that are forced because of the domination condition, and those that are forced because of the location condition.

\begin{definition}\label{defforcedver} Let $D$ be a locatable digraph. A vertex $v$ of $D$ is called \emph{domination-forced} if there exists a vertex $w$, such that $v$ is the unique in-neighbour of $w$. Vertex $v$ is called \emph{location-forced} if there exist two distinct vertices $x$ and $y$, such that $N^-(x)\ominus N^-(y)=\{v\}$ (where $A\ominus B$ denotes the symmetric difference of two sets $A$ and $B$). A vertex $v$ is called \emph{double-forced}, if either it is both domination-forced and location-forced, or it is location-forced because of two different pairs of vertices.
\end{definition}

We can observe the following.

\begin{proposition}\label{prop:all-forced}
	If there is a vertex $v$ in a locatable digraph $D$ which is neither domination-forced nor location-forced, then $V(D)\setminus\{v\}$ is an OLD set of $D$.
\end{proposition}
\begin{proof}
	Since $v$ is not domination-forced, every vertex of $D$ has an in-neighbour in $V(D)\setminus\{v\}$. Moreover, since $v$ is not location-forced, for every pair $z,w$ of distinct vertices in $D$, there is a vertex in $V(D)\setminus\{v\}$ in the symmetric difference $N^-(z)\ominus N^-(w)$, which therefore distinguishes $z$ and $w$.
\end{proof}

Proposition~\ref{prop:all-forced} implies that in any extremal digraph $D$, 
every vertex is domination-forced or location-forced (or both). (In fact, we will show in Proposition~\ref{propblueblue} that 
no vertex of $D$ could be both domination-forced and location-forced.)

We get a direct corollary of Proposition~\ref{prop:all-forced}, which will be used several times in the proofs of Section~\ref{sec:trees}.

\begin{corollary}\label{cor-notdomforc-locforc}
	Let $D$ be an extremal digraph. 
	If a vertex is not domination-forced (resp. location-forced), then it must be location-forced (resp. domination-forced).
\end{corollary}

Before proving our characterization, we will use the following celebrated theorem of Bondy, which is important for our line of work (see for example~\cite{FNP13} and references therein).

\begin{theorem}[Bondy's Theorem \cite{B72}]\label{thm:bondy}
	Let $V$ be an $n$-set, and $\mathcal{A}=\{\mathcal{A}_1,\mathcal{A}_2,\ldots,\mathcal{A}_n\}$ be a family of $n$ distinct subsets of $V$. There is an $(n-1)$-subset $X$ of $V$ such that the sets $\mathcal{A}_1\cap X, \mathcal{A}_2\cap X, \mathcal{A}_3\cap X,\ldots, \mathcal{A}_n\cap X$ are still distinct.
\end{theorem}

\begin{corollary}\label{cor:one-dom-forced}
	Every locatable digraph $D$ of order $n$ has at most $n-1$ location-forced vertices.
\end{corollary}
\begin{proof}
	Construct from digraph $D$ the set system with $V(D)$ as its $n$-set and where the $A_i$'s are all the open in-neighbourhoods of vertices of $D$. Theorem~\ref{thm:bondy} implies that there is one vertex such that removing it does not create two same open in-neighbourhoods. In other words, this vertex is not location-forced.
\end{proof}

\subsection{Forcing arcs}

When a vertex $x$ is forced, either it is the unique in-neighbour of some vertex $y$, or there are two vertices $y,z$ such that $x$ is the only vertex in the symmetric difference between $N^-(y)$ and $N^-(z)$, and $x\in N^-(y)$. In some sense, the arc $xy$ is remarkable in that respect. We highlight such arcs as follows.

\begin{definition}\label{def:red-edges}
	Let $D=(V,A)$ be an extremal digraph.
	Then the arc $xy\in A$ is called a \emph{forcing arc} if either $N^-(y)=\{x\}$ or there is a vertex $z\in V$ such that $N^-(y)\setminus N^-(z) =\{x\}$. A \emph{forcing cycle} is a cycle all whose arcs are forcing arcs.
\end{definition}

Note that a vertex is forced if and only if it is the tail of a forcing arc. Thus, if $D$ has only forced vertices, every vertex has a forcing outgoing arc.

The next lemma is important for our study.

\begin{lemma}\label{lem-nonrededge}
	Let $D$ be an extremal digraph. 
	Let $x$ be an arbitrary vertex of $D$ and let $D'$ be the digraph obtained from $D$ by deleting all non-forcing arcs of $D$, which have $x$ as their tails. Then, $D'$ is locatable and extremal. 
	Moreover, if $x$ is not location-forced, then $D$ and $D'$ have the same sets of forcing arcs.
\end{lemma}
\begin{proof}
	First, we prove that $D'$ is locatable. Towards a contradiction, suppose that there exist two in-twins: vertices $y$ and $z$ such that $N_{D'}^-(y)=N_{D'}^-(z)$. Since $D$ is locatable, we conclude that $x$ had exactly one of $y$ and $z$ as an out-neighbour in $D$; without loss of generality suppose that $xy$ is an arc of $D$. But then, $xy$ would be a forcing arc and it would not have been deleted from $D$, a contradiction. Moreover, assume that there is a vertex $t$ of in-degree~0 in $D'$. Then, $x$ must have been the only in-neighbour of $t$ in $D$, but then the arc $xt$ would be forcing and $t$ would have in-degree~1 in $D'$, a contradiction. Therefore, $D'$ is locatable.
	
	We will now show that every vertex of $D'$ is the tail of a forcing arc, implying that $\gamma_{OL}(D')=n$, and that if $x$ is not location-forced, then $D$ and $D'$ have the same sets of forcing arcs.
	
	By Proposition~\ref{prop:all-forced}, all vertices of $D$ are either domination-forced or location-forced. By deleting all non-forcing arcs of $D$ which have $x$ as their tail, it is clear that all domination-forced vertices remain domination-forced, and each forcing arc $xy$ having a domination-forced vertex as its tail remains forcing. Thus, to complete the proof, it remains only to consider location-forced vertices and those forcing arcs of $D$ that have a location-forced vertex as their tails.
	
	To this end, consider a location-forced vertex $t$ in $D$, that is, there are two vertices $u$ and $v$ such that $N_D^-(u)\setminus N_D^-(v)=\{t\}$; thus, $tu$ is a forcing arc in $D$. If $x$ is neither an in-neighbour of $u$ nor an in-neighbour of $v$, or if $x=t$, then in $D'$ we still have $N_{D'}^-(u)\setminus N_{D'}^-(v)=\{t\}$ and $tu$ is still a forcing arc in $D'$. Otherwise, $x$ is an in-neighbour of $u$ or $v$ and $x\neq t$, thus, $x$ must be an in-neighbour of both $u$ and $v$ in $D$ (possibly, $x=u$ or $x=v$). If both arcs $xu$ and $xv$ are not forcing or if both are forcing, then again in $D'$, $N_{D'}^-(u)\setminus N_{D'}^-(v)=\{t\}$ and $tu$ is still a forcing arc in $D'$. If $xu$ is forcing and $xv$ is not forcing in $D$, since $u$ has both $x$ and $t$ as in-neighbours, $xu$ is a location-forcing arc and there is a vertex $w$ with $N^-(u)\setminus N^-(w)=\{x\}$. We note that since $x\neq t$ and $t\in N_D(u)$, we have $tw\in A(D)$. But then, in $D'$, we have $N_{D'}^-(w)\setminus N_{D'}^-(v)=\{t\}$. Thus, $t$ is still location-forced in $D'$. Though the forcing arc $tu$ is no longer forcing in $D'$, now the arc $tw$ is forcing in $D'$ (and in that case we had that $x$ is location-forced). 
	
	Assume finally that $xv$ is forcing and $xu$ is not forcing in $D$. If $xv$ is forcing because of domination, it means that $x$ is the unique in-neighbour of $v$ in $D$, and thus $u$ is dominated only by $x$ and $t$; in $D'$, $u$ is dominated only by $t$, and thus in $D'$ the arc $tu$ remains forcing and $t$ is now domination-forced. Otherwise, $xv$ is forcing because of location: there is a vertex $w$ such that $N_{D}^-(v)\setminus N_{D}^-(w)=\{x\}$. We note that since $xu$ is non-forcing arc in $D$, and  $N_{D}^-(u)\setminus N_{D}^-(w)=\{x,t\}$, we conclude that in $D'$, we have $N_{D'}(u)\setminus N_{D'}(w)=\{t\}$. Thus, the arc $tu$ stays forcing in $D'$ and $t$ is still location-forced.
	
	This means that each vertex of $D'$ is either domination-forced or location-forced, and thus, we conclude that $\gamma_{OL}(D')=n$. Moreover, the only case where $D'$ and $D$ had different sets of forcing arcs occurred when $x$ was location-forced, as claimed.
\end{proof}

We next show that in an extremal digraph, 
no two forcing arcs can have the same tail.

\begin{proposition}\label{propblueblue}
	No extremal digraph contains a double-forced vertex.
\end{proposition}
\begin{proof}
	We prove this by induction on $n$. We can assume $D$ is connected, as it suffices to prove the claim for each connected component. If $n=1$, the only locatable digraph has a single vertex with a loop, for which the claim is clearly true. If $n=2$, one can check that there are three connected locatable digraphs of order~2 and in fact they all have OLD number~2 (see Figure~\ref{fig:order2}). For each of them the claim is true.
	
	Let $n\geq 3$ and assume the result is true for all digraphs $D$ of order $m<n$ with $\gamma_{OL}(D)=m$. 		
	Towards a contradiction, suppose that there is a digraph $D$ of order $n$ with $\gamma_{OL}(D)=n$, which contains a double-forced vertex.
	Among all such digraphs of order $n$, let $D=(V, A)$ be a digraph which has the smallest number of arcs.
	
	Let $z\in V$ be a double-forced vertex of $D$. By Corollary~\ref{cor:one-dom-forced}, there is a vertex $x$ in $D$ which is not location-forced, hence it is domination-forced. So, there is a vertex $y\in V$ with $N^-(y)=\{x\}$ and $xy$ is a forcing arc (possibly $x=y$ and the arc is a loop).
	Since $x$ is not location-forced, and there cannot be another vertex that has $x$ as its unique in-neighbour (otherwise it would be an in-twin of $y$, contradicting the fact that $D$ is locatable), we conclude that $xy$ is the unique forcing arc which has $x$ as its tail.
	
	Now, we claim that $N^+(x)=\{y\}$. Indeed, otherwise, we can delete all non-forcing arcs from $D$ which have $x$ as their tails, to obtain a new digraph $D'$; by Lemma~\ref{lem-nonrededge} applied to $D$ and $x$, which is not location-forced, we have $\gamma_{OL}(D')=n$ and the set of forcing arcs of $D$ and $D'$ are the same. 
	Thus, all double-forced vertices of $D$ remain double-forced in $D'$, in particular, there is at least one double-forced vertex in $D'$. But $D'$ has at least one arc less than $D$, which contradicts the minimality of $D$ in terms of the number of its arcs. (Moreover, if $D'$ is not connected, we contradict the induction hypothesis applied to a connected component of $D'$ containing a double-forced vertex.) 
	Therefore, we have $N^+(x)=\{y\}$ as claimed. As $n\geq 3$ and $D$ is connected, this implies that $x\neq y$.
	
	Now, let $D''$ be the digraph obtained from $D$ by contracting the arc $xy$. That is, we delete $x$ and $y$ and add a new vertex $v_{xy}$ that represents both $x$ and $y$. Then, for each arc whose head is $x$ or $y$ (except the arc $xy$), we add an arc from its tail to $v_{xy}$; similarly, for each arc whose tail is $x$ or $y$ (except the arc $xy$), we add an arc from $v_{xy}$ to its head. Then, every domination-forced vertex of $D$ (except $x$) remains domination-forced in $D''$. Moreover, every location-forced vertex of $D$ remains location-forced in $D''$ (note that $v_{xy}$ is domination-forced in $D''$ if $y$ was domination-forced in $D$, and is location-forced if $y$ was location-forced in $D$). Hence, $\gamma_{OL}(D'')=n-1$ and the vertex $z$ (or $v_{xy}$ if $z=y$) is double-forced in $D''$, which contradicts the induction hypothesis. Thus, $D$ does not exist, a contradiction which completes the proof.
\end{proof}

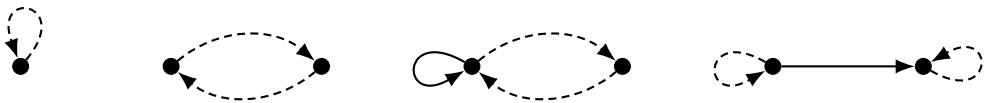
\begin{figure}[!htpb]
	\centering
	\begin{tikzpicture}
		
		\node[small node](x00) at (-2,0)    {};
		\path[thick,densely dashed,-{Latex[scale=1.1]}] (x00) edge[out=50,in=110,loop, min distance=10mm] node {} (x00);
		
		\node[small node](x1) at (0,0)    {};
		\node[small node](y1) at (2,0)    {};
		\node[small node](x2) at (4,0)    {};
		\node[small node](y2) at (6,0)    {};
		\node[small node](x3) at (8,0)    {};
		\node[small node](y3) at (10,0)    {};
		
		\path[thick,densely dashed,-{Latex[scale=1.1]}] (x1) edge[out=40,in=140,min distance=5mm] node {} (y1);
		\path[thick,densely dashed,-{Latex[scale=1.1]}] (y1) edge[out=220,in=320,min distance=5mm] node {} (x1);
		\path[thick,densely dashed,-{Latex[scale=1.1]}] (x2) edge[out=40,in=140,min distance=5mm] node {} (y2);
		\path[thick,densely dashed,-{Latex[scale=1.1]}] (y2) edge[out=220,in=320,min distance=5mm] node {} (x2);
		\path[thick,-{Latex[scale=1.1]}] (x2) edge[out=150,in=210,loop, min distance=10mm] node {} (x2);
		\path[thick,-{Latex[scale=1.1]}] (x3) edge node {} (y3);
		\path[thick,densely dashed,-{Latex[scale=1.1]}] (x3) edge[out=150,in=210,loop, min distance=10mm] node {} (x3);
		\path[thick,densely dashed,-{Latex[scale=1.1]}] (y3) edge[out=330,in=30,loop, min distance=10mm] node {} (y3);
		
	\end{tikzpicture}\centering
	\caption{The four connected locatable digraphs of order~1 and~2. Forcing arcs are dashed.}\label{fig:order2}\label{fig:order1,2}
\end{figure}

\subsection{Structural properties of extremal digraphs}

\begin{theorem}\label{thmmain}
	Let $D$ be a digraph of order $n$ and $D'$ be the subdigraph of $D$ induced by the forcing arcs of $D$. Then, $D$ is extremal 
	if and only if $D'$ is the disjoint union of directed cycles that spans the whole vertex set of $D$.
\end{theorem}
\begin{proof}	
	Assume $\gamma_{OL}(D)=n$. By repeated use of Lemma~\ref{lem-nonrededge}, we deduce that $D'$ is locatable and $\gamma_{OL}(D')=n$. Thus, each vertex of $D'$ is forced, has at least one in-neighbour, and at least one out-neighbour. In fact, by Proposition~\ref{propblueblue}, each vertex of $D'$ has exactly one out-neighbour. Thus, there is a total of $n$ arcs in $D'$, and so, every vertex of $D'$ has exactly one in-neighbour and one out-neighbour, and $D'$ is the disjoint union of directed cycles.
	
	Conversely, if $D'$ is the disjoint union of directed cycles, then $\gamma_{OL}(D')=n$. By Lemma~\ref{lem-nonrededge}, $D$ and $D'$ have the same set of forced vertices, thus $\gamma_{OL}(D)=n$.
\end{proof}

By Theorem~\ref{thmmain}, every vertex $v$ of a digraph $D$ of order $n$ with $\gamma_{OL}(D)=n$ has a unique outgoing and a unique incoming forcing arc (possibly they are the same if $v$ has a forcing loop).

\begin{definition}\label{def-f-f+}
	For a vertex $v$ of an extremal digraph $D$, we denote by $f^-(v)$ and $f^+(v)$ the unique in-neighbour and out-neighbour of $v$, respectively, corresponding to the two unique incoming and outgoing forcing arcs incident with $v$ (if $v$ has a forcing loop, we have $f^+(v)=f^-(v)=v$).
\end{definition}

Theorem~\ref{thmmain} implies that in an extremal digraph $D$, 
every vertex appears in a directed cycle, thus we get the following corollary.

\begin{corollary}
	Let $D$ be a digraph of order $n$ containing a source or a sink. Then, $\gamma_{OL}(D)\leq n-1$.
\end{corollary}

We also get the following corollary.

\begin{corollary}\label{corlacatable}
	If each vertex of a digraph $D$ is forced, then $D$ is locatable.
\end{corollary}
\begin{proof}
	By Theorem~\ref{thmmain}, every vertex of $D$ has an in-neighbour. Assume by contradiction that $D$ contains two vertices $x$ and $y$ with the same in-neighbourhood. By Theorem~\ref{thmmain}, $x$ has a forcing incoming arc, $tx$. Thus, there is an arc $ty$ but by Proposition~\ref{propblueblue} $ty$ is not forcing. Hence, $t$ is not the only in-neighbour of $y$, and $x,y$ have at least two in-neighbours. Thus, $t$ is location-forced and there is a vertex $z$ with $N^-(x)\setminus N^-(z)=\{t\}$. But this implies $N^-(y)\setminus N^-(z)=\{t\}$ and the arc $ty$ should be forcing, contradicting Proposition~\ref{propblueblue}.
\end{proof}

\begin{definition}\label{def:H(D)}
	Given a digraph $D$, we define the digraph $\mathcal{H}(D)$ on vertex set $V(D)$, where $x$ has an arc to $y$ if and only if there exists a vertex $v$ of $D$ that is location-forced, with $N^-(x)=N^-(y)\setminus\{v\}$ (possibly, $v=y$, in which case $y$ has a forcing loop; if $v=x$, then $x$ has no loop but there is a forcing arc from $x$ to $y$ in $D$).    
\end{definition}

Such a construction was previously defined in~\cite{FP12} in the context of identifying codes. We will now give some properties of $\mathcal{H}(D)$ when $D$ is an extremal digraph.

\begin{theorem}\label{thm:H(D)}
	Let $D$ be an extremal digraph.
		Then, $\mathcal{H}(D)$ is the disjoint union of rooted directed trees, where for each root $r$, $f^-(r)$ is domination-forced in $D$ 
		(and thus $r$ has only one in-neighbour in $D$), and for each other vertex $v$, $f^-(v)$ is location-forced in $D$ (and thus, $v$ has an in-neighbour in $\mathcal{H}(D)$).

	\end{theorem}
	\begin{proof}
		Since an arc $xy$ in $\mathcal{H}(D)$ implies that the in-neighbourhood of $x$ is strictly smaller than that of $y$, it is clear that $\mathcal{H}(D)$ is acyclic. Moreover, if some vertex $x$ has two in-neighbours $y,z$ in $\mathcal{H}(D)$, since $f^-(x)$ is unique and by the definition of $\mathcal{H}(D)$, then we would have that $N^-(y)=N^-(x)\setminus\{f^-(x)\}=N^-(z)$, and thus $y,z$ would be in-twins, contradicting the fact that $D$ is locatable. 
		Thus, $\mathcal{H}(D)$ is acyclic and each vertex has at most one in-neighbour, hence $\mathcal{H}(D)$ is the disjoint union of rooted directed trees as claimed.
		
		By Theorem~\ref{thmmain}, every vertex $v$ of $D$ has an incoming forcing arc from $f^-(v)$. By the definition of $\mathcal{H}(D)$, if $v$ is not a root of a tree of $\mathcal{H}(D)$, $f^-(v)$ is location-forced. If $r$ is a root of a tree of $\mathcal{H}(D)$, then by the definition of $\mathcal{H}(D)$, $f^-(r)$ is not location-forced, and since $D$ is extremal by Corollary \ref{cor-notdomforc-locforc}, $f^-(r)$ is domination-forced.
		
		By the definition of $\mathcal{H}(D)$, each vertex $v$ with an in-neighbour in $\mathcal{H}(D)$ has an incoming forcing arc $wv$ where $w=f^-(v)$ is location-forced. This completes the proof.
	\end{proof}

	Using Theorem~\ref{thmmain} and Theorem~\ref{thm:H(D)}, one can show how all extremal digraphs can be built, as follows.
	
	\begin{theorem}\label{thm:construct}
		For any locatable digraph $D$ of order $n$, we have $\gamma_{OL}(D)=n$ if and only if $D$ can be constructed as follows.
		\begin{enumerate}
			\item First, choose a decomposition of $n$ as a sum of positive integers $n_1,\ldots n_k$, corresponding to the orders of the directed cycles $C_1,\ldots,C_k$ consisting of all forced arcs of $D$, and create the corresponding cycles.
			
			\item Next, choose a partition of $V(D)$ into a set $V_d$ of domination-forced vertices and a set $V_l$ of location-forced vertices, with $|V_d|\geq 1$.
			
			\item Then, construct $\mathcal{H}(D)$ as a collection of vertex-disjoint rooted directed trees (note that such a tree may consist of a single vertex), as follows. The roots of the trees are precisely the out-neighbours of the vertices in $V_d$. Moreover, for any vertex $x$ of $V_l$, its out-neighbour $f^+(x)$ has an in-neighbour in $\mathcal{H}(D)$.
			
			\item Finally, for each rooted directed tree $T$ of $\mathcal{H}(D)$ and every vertex $v$ of $T$, we create an arc from $f^-(v)$ to all descendants of $v$ in $T$.
		\end{enumerate}	
	\end{theorem}
	\begin{proof}
		Assume that $D$ is extremal. By Theorem~\ref{thmmain}, the subdigraph $D'$ of $D$ induced by the forcing arcs of $D$ is the disjoint union of directed cycles that spans the whole vertex set of $D$. This corresponds to the first step of the construction. Every vertex is forced, and by Proposition~\ref{propblueblue}, no vertex is double-forced. Thus, there is a partition of $V(D)$ into the set $V_d$ of domination-forced vertices and the set $V_l$ of location-forced vertices. This is Step~2 of the construction. Moreover, by Corollary~\ref{cor:one-dom-forced}, $|V_d|\geq 1$. By Theorem~\ref{thm:H(D)}, the digraph $\mathcal{H}(D)$ is a collection of vertex-disjoint rooted directed trees where the roots of the trees are precisely the out-neighbours of the vertices in $V_d$. Moreover, for any vertex $x$ of $V_l$, its out-neighbour $f^+(x)$ has an in-neighbour in $\mathcal{H}(D)$ (i.e. it is not a root of a tree of $\mathcal{H}(D)$). This corresponds to Step~3 of the construction. Now, the arcs of $D$ comply with the definition of $\mathcal{H}(D)$: for any vertex $x$ of a tree $T$ in $\mathcal{H}(D)$, the in-neighbourhood in $D$ of each descendant of $x$ in $T$ contains the in-neighborhood of $x$ in $D$, and moreover, for any arc $xy$ of $\mathcal{H}(D)$, $N^-_D(x)=N^-_D(y)\setminus\{f^-(y)\}$. Since the root $r$ of $T$ has only $f^-(r)$ as an in-neighbor, there are no further arcs in $D$ incoming towards a vertex of $T$ (otherwise there would be a similar arc towards the root $r$, contradicting the fact that it has only one in-neighbour). Thus, there are in fact no more arcs in $D$ than the ones following the structure of $\mathcal{H}(D)$, and thus Step~4 completes the description of $D$.

		Conversely, if $D$ is constructed in this way, the digraph is clearly locatable and each vertex is forced, and thus $D$ is indeed extremal.
	\end{proof}
	
	An example of the construction of Theorem~\ref{thm:construct} is depicted in Figure~\ref{fig:construct}. Figure~\ref{fig:construct}(a) shows the choice of the directed cycles formed by the forcing arcs (two cycles $(1,3,2)$ and $(4)$) as well as the partition into $V_d=\{1\}$ and $V_f=\{2,3,4\}$). Figure~\ref{fig:construct}(b) shows the set of rooted diected trees $\mathcal{H}(D)$ (in this case, it consists of a single tree $T$ rooted at vertex $3$, which is the out-neighbour of vertex $1$, the only vertex in $V_d$). (Note that $\mathcal{H}(D)$ is not a subdigraph of $D$.) Finally, Figure~\ref{fig:construct}(c) shows the resulting extremal digraph obtained by adding to the set of directed cycles, for each vertex $v$ in $T$, an arc from $f^-(v)$ to all descendants of $v$ in $T$. That is, we add arcs from $1=f^-(3)$ to $1$, $2$ and $4$ and from $3=f^-(2)$ to $1$ (and no further arcs since $1$ and $4$ have no descendants in $T$).

	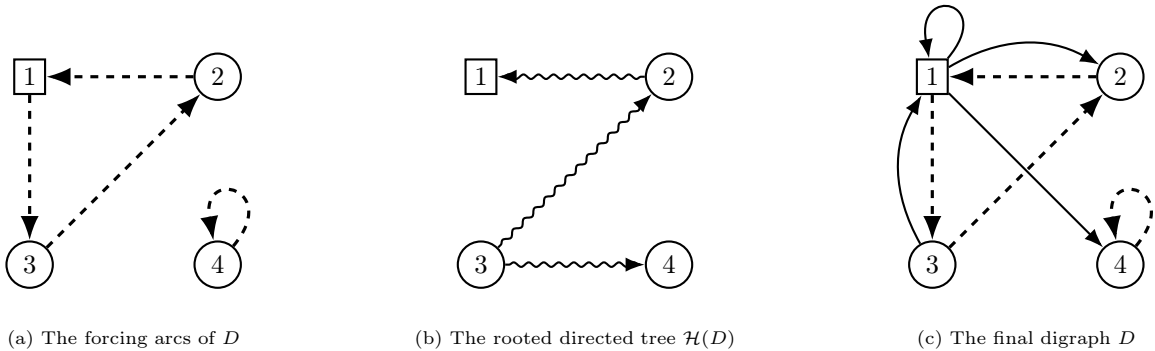
\begin{figure}[!htpb]
		\centering
		\begin{tikzpicture}
			[
			-{Latex[scale=1.1]},
			node distance = 2cm, 
			thick 
			]
			
			\tikzset{every state}=[draw = black, thick, fill = white,minimum size = 1mm]
			\node[shape=rectangle,draw=black] (1) at (0,2.5) {1};
			\node[shape=circle,draw=black] (3) at (0,0) {3};
			\node[shape=circle,draw=black] (2) at (2.5,2.5) {2};
			\node[shape=circle,draw=black] (0) at (2.5,0) {4};
			
			\node at (1.25,-1) {{\scriptsize (a) The forcing arcs of $D$}};
			
			\path [draw,very thick, dashed, color=black](2) edge node[left] {} (1);
			\path [draw,very thick, dashed, color=black](1) edge node[left] {} (3);
			\path [draw,very thick, dashed, color=black](3) edge node[left] {} (2);
			\path[draw,very thick,dashed] (0) edge[out=50,in=100,loop, min distance=10mm] node {} (0);
			
			\node[shape=rectangle,draw=black] (4) at (6,2.5) {1};
			\node[shape=circle,draw=black] (6) at (6,0) {3};
			\node[shape=circle,draw=black] (5) at (8.5,2.5) {2};
			\node[shape=circle,draw=black] (7) at (8.5,0) {4};
			
			\node at (7.25,-1) {{\scriptsize (b) The rooted directed tree $\mathcal{H}(D)$}};

			\tikzset{decoration={snake,amplitude=.3mm,segment length=2mm,
					post length=0.4mm,pre length=0.5mm}}
			\draw[decorate] (5) -- (4);
			\draw[decorate] (6) -- (5);
			\draw[decorate] (6) -- (7);
			
			\node[shape=rectangle,draw=black] (8) at (12,2.5) {1};
			\node[shape=circle,draw=black] (9) at (12,0) {3};
			\node[shape=circle,draw=black] (10) at (14.5,2.5) {2};
			\node[shape=circle,draw=black] (11) at (14.5,0) {4};
			
			\node at (13.25,-1) {{\scriptsize (c) The final digraph $D$}};
			
			\path [draw,very thick, dashed, color=black](10) edge node[left] {} (8);
			\path [draw,very thick, dashed, color=black](8) edge node[left] {} (9);
			\path [draw,very thick, dashed, color=black](9) edge node[left] {} (10);
			
			\path[draw,thick] (8) edge[bend left] (10);
			\path [draw, color=black](8) edge node[bend left] {} (11);
			\path[draw,thick] (9) edge[bend left] (8);
			\path[draw,thick] (8) edge[out=50,in=100,loop, min distance=10mm] node {} (8);
			\path[draw,very thick,dashed] (11) edge[out=50,in=100,loop, min distance=10mm] node {} (11);
		\end{tikzpicture}
		\caption{An example of the construction from Theorem~\ref{thm:construct}. The square vertex is the only one in $V_d$; the circled vertices are those in $V_l$; the dashed arcs are the forcing arcs; the wriggled arcs are those of $\mathcal H(D)$.}
		\label{fig:construct}
	\end{figure}

	\section{New proofs of known results}\label{sec:known}

	In this section, we show that the contents of the previous section enable us to give new proofs for already known results.
	
	\subsection{A new proof of Theorem~\ref{thm:IDcodes-n-1}}
	
	Recall the statement of Theorem~\ref{thm:IDcodes-n-1}.
	
	\begin{theorem*}[Theorem~\ref{thm:IDcodes-n-1}]
		For a connected, symmetric and reflexive locatable digraph $D$ of order $n$, $\gamma_{OL}(D)=n$ if and only if $n=1$.
	\end{theorem*}
	
	We note that Corollary~\ref{cor:one-dom-forced}, which is directly derived from Theorem~\ref{thm:bondy} by Bondy~\cite{B72}, in fact implies Theorem~\ref{thm:IDcodes-n-1} (this gives a different proof than the one from~\cite{GM07}).
	
	\begin{proof}[Proof of Theorem~\ref{thm:IDcodes-n-1}]
		Let $D$ be a connected, reflexive, symmetric and locatable digraph of order $n$. If $\gamma_{OL}(D)=n$, by Proposition~\ref{prop:all-forced}, every vertex of $D$ is either location-forced or domination-forced. By Corollary~\ref{cor:one-dom-forced}, $D$ has at most $n-1$ location-forced vertices, and so, it has at least one domination-forced vertex. However, since $D$ is reflexive and symmetric, a domination-forced vertex of $D$ is necessarily a vertex with no neighbours other than itself. Since $D$ is connected, we must have $n=1$.
	\end{proof}

	\subsection{A new proof of Theorem~\ref{thm:IDcodes-oriented}}
	
	Our tools can be used to give a new proof of Theorem~\ref{thm:IDcodes-oriented} from~\cite{FNP13} (the original proof uses induction), whose statement we recall below.
	
	\begin{theorem*}[Theorem~\ref{thm:IDcodes-oriented}]
		For a connected and reflexive locatable digraph $D$ of order $n$ without directed 2-cycles, $\gamma_{OL}(D)=n$ if and only if the digraph obtained from $D$ by removing all loops is the transitive closure of a rooted directed tree.
	\end{theorem*}

	\begin{proof}[Proof of Theorem~\ref{thm:IDcodes-oriented}]
		It is not difficult to see that if $D$ is obtained from the transitive closure of a rooted directed tree by adding a loop to each vertex, then $\gamma_{OL}(D)=n$ as the root of the tree is domination-forced, and each vertex is location-forced to locate itself from its parent in the tree.
		
		For the other direction, let $D$ be a connected reflexive locatable digraph of order $n$ with no directed 2-cycle, and assume that $\gamma_{OL}(D)=n$.
		
		First of all, we claim that the forcing arcs in $D$ are exactly its loops. Assume by contradiction that it is not the case, and there is a forcing arc from $x$ to $y$ with $x\neq y$. Then, there is a vertex $z$ such that $N^-(y)\setminus N^-(z)=\{x\}$ (thus, $z\notin\{x,y\}$ since $x$ is an in-neighbour of both $x$ and $y$ since there is a loop at $x$). Since there is a loop at both $y$ and $z$, there is an arc from $y$ to $z$ and vice-versa, contradicting the fact that there is no directed 2-cycle in $D$. Thus, each forcing arc is a loop, and by Theorem~\ref{thmmain}, the set of forcing arcs of $D$ is exactly its set of loops.
		
		Now, consider the digraph $\mathcal{H}(D)$ from Definition~\ref{def:H(D)}. By Theorem~\ref{thm:H(D)}, it consists of a disjoint union of rooted directed trees. Since every vertex is dominated by itself through its loop, every domination-forced vertex is the root of one of the directed trees of $\mathcal{H}(D)$. Consider a location-forced vertex $x$ of $D$, and assume its in-neighbour in $\mathcal{H}(D)$ is $y$. By the previous paragraph we have $f^-(x)=x$ and thus, since $y$ has a loop, we must have the arc $yx$ in $D$ as well. Thus, $\mathcal{H}(D)$ is in fact a subdigraph of $D$. Moreover, for any two vertices $x,y$ in the same rooted directed tree of $\mathcal{H}(D)$, where $x$ is a descendant of $y$, we have the arc $yx$ in $D$.

		Moreover, we claim that there is a unique tree in $\mathcal{H}(D)$. For a contradiction, suppose there are at least two of them (each of which has a domination-forced vertex as its root). Recall that $\mathcal{H}(D)$ is a subdigraph of $D$. Since $D$ is connected, there must be two trees $T_1$ and $T_2$ with an arc say, from a vertex $x_1$ of $T_1$ to a vertex $x_2$ of $T_2$. But then, since the in-neighborhoods of vertices of $T_2$ only differ by vertices inside $T_2$, $x_1$ must be an in-neighbour of all vertices of $T_2$ (including the root of $T_2$), and thus the root of $T_2$ is in fact not domination-forced, a contradiction.
		
		This shows that $D$ is obtained from the transitive closure of a rooted directed tree by adding a loop to each vertex, as claimed.
	\end{proof}
	
	\subsection{A new proof of Theorem~\ref{thm:half-graphs}}
	We next give a new proof using our structural theorems, that for every connected locatable symmetric and loop-free digraph of order $n$ with $\gamma_{OL}(D)=n$, the underlying graph of $D$ is a half-graph (see below). We recall the definition of a half-graph: for any integer $k\geq 1$, the half-graph $H_k$ is the undirected bipartite graph on vertex sets $\{v_1,\ldots,v_k\}$ and $\{w_1,\ldots,w_k\}$, with an edge between $v_i$ and $w_j$ if and only if $i\leq j$.
	
	\begin{theorem*}[{Theorem~\ref{thm:half-graphs}}]
		For a connected, symmetric and loop-free locatable digraph $D$ of order $n$, $\gamma_{OL}(D)=n$ if and only if the underlying graph of $D$ is a half-graph.
	\end{theorem*}

	\begin{proof}[Proof of Theorem~\ref{thm:half-graphs}]
		Assume that $D$ is a connected, locatable, loop-free and symmetric digraph of order $n$ with $\gamma_{OL}(D)=n$. 
		By Theorem~\ref{thmmain}, we know that the set of forcing arcs of $D$ induces a disjoint union of directed cycles.

		First we show that all these directed cycles are, in fact, 2-cycles. Towards a contradiction, assume this is not the case, and let $C$ be a directed cycle of forcing arcs of length other than~2. Since $D$ is loop-free, there are no forcing loops, and so, $C$ has length at least~3. Let $c_1,c_2,\ldots,c_k$ be the vertices of $C$, ordered along the natural orientation of $C$. Since $D$ is symmetric, each vertex of $C$ has at least two in-neighbours, thus no vertex of $C$ is domination-forced, and hence they are all location-forced. Thus, for each vertex $c_i$ of $C$, there is a vertex $c'_i$ such that $N^-(c_i)\setminus\{c_{i-1}\}=N^-(c'_i)$. Consider $i=2$. Since $D$ is symmetric, we have $c_{3}$ which has an arc to $c_2$, and thus, there are symmetric arcs between $c_{3}$ and $c'_2$ as well. Thus, since $N^-(c_{3})\setminus\{c_{2}\}=N^-(c'_{3})$, there must also exist symmetric arcs between $c'_2$ and $c'_{3}$. However, since $N^-(c_2)\setminus\{c_{1}\}=N^-(c'_2)$ and $c_{1}\neq c_{3}$, there must be symmetric arcs between $c_2$ and $c'_{3}$. But this contradicts the fact that $N^-(c_{3})\setminus\{c_{2}\}=N^-(c'_{3})$, and proves the claim that all forced-cycles are 2-cycles.
		
		By Corollary~\ref{cor:one-dom-forced}, we know that $D$ contains at least one domination-forced vertex. Let $v_1$ be a domination-forced vertex of $D$ and $v_1=f^-(u_1)$. Since every forced-cycle of $D$ is of length~$2$, we conclude that $u_1=f^-(v_1)$. Now, $d^-(u_1)=1$, since $v_1$ is domination-forced. Since $D$ is symmetric, we have $d^+(u_1)=1$. If $u_1$ is also domination-forced, since $D$ is connected, then $D$ is of order~$2$ and its underlying graph is the half-graph of order~2. Otherwise, $u_1$ is location-forced and thus there is a vertex $v_2$ such that $N^-(v_1)\setminus N^-(v_2)=\{u_1\}$.
		
		Now, let $u_2=f^-(v_2)$, and again, since the forced cycles of $D$ are all 2-cycles, we also have $v_2=f^-(u_2)$. Since $N^-(v_1)\setminus N^-(v_2)=\{u_1\}$, we also have the arc $u_2v_1$ and (since $D$ is symmetric) the arc $v_1u_2$. 
		
		If $u_2$ is domination-forced, then $v_1$ and $v_2$ have no additional in-neighbours. Thus, the only in-neighbour of $v_2$ is $u_2$, which has at least two in-neighbours, and thus, $v_2$ cannot be domination-forced. Thus, $v_2$ is location-forced, and since $f^-(u_2)=v_2$, there is a vertex $u_3$ such that $N^-(u_2)\setminus N^-(u_3)=\{v_2\}$. Thus, $u_3$ must have $v_1$ as an in-neighbour, and in fact we have $u_1=u_3$ and there are no other vertices in $D$. Now, we are done since the underlying graph of $D$ is a half-graph of order~4.
		
		Otherwise, $u_2$ is location-forced, and since $u_2=f^-(v_2)$, there is a vertex $v_3$ such that $N^-(v_2)\setminus N^-(v_3)=\{u_2\}$. We can continue this process, building disjoint pairs of vertices $(u_i,v_i)$ forming the forcing 2-cycles of $D$, where $u_i$ has an outgoing arc and an incoming arc to and from each vertex $v_j$, with $j\leq i$. This goes on until we reach a domination-forced vertex $u_k$. Then, the process stops. The vertex set of th obtained graph is $\{u_1,\ldots,u_k\}\cup\{v_1,\ldots,v_k\}$, and there are two symmetric arcs between $u_i$ and $v_j$ if and only if $i\leq j$. Thus, the underlying graph of $D$ is precisely a half-graph of order $2k$, which completes the proof.
	\end{proof}

	\section{A recursive and constructive characterization of extremal di-trees}\label{sec:trees}
	
	In this section, we characterize extremal di-trees, that is, (connected) extremal digraphs whose underlying graph is a tree. We are going to give a recursive construction for all of these digraphs. This characterization is more precise than the one from the more general Theorem~\ref{thm:construct} that holds for all extremal digraphs, and in particular, it enables us to easily construct all extremal di-trees of order $n$ from the ones of orders $n-2$ and $n-1$ using simple operations.

	We start with the following definitions.
	
	\begin{definition} \label{def-ditree}
		For a positive integer $n$, we define $\mathcal{T}_n$ as the set of extremal di-trees, that is, all locatable di-trees 
		$D$ of order $n$ with $\gamma_{OL}(D)=n$. 
	\end{definition}
	
	We note that when $D$ is di-tree, then every forcing cycle of $D$ is either of length $2$ or of length $1$. This implies that for every vertex $v$ of $D$, $f^-(v)=f^+(v)$ and $f^-(f^-(v))=v$.

	For an undirected graph $G$, we say that an induced path with vertices $u_0, u_1, \ldots , u_n$ of $G$ is
	a {\em pendant path} of length $n$ of $G$, if
	$deg(u_1)= \ldots = deg(u_{n-1})=2$ and $deg(u_n)=1$
	(note there is no requirement on the degree of $u_0$). 
	
	\begin{lemma}\label{lemtd-<2}
		Let $D\in \mathcal{T}_n$ and $T$ be the underlying tree of $D$. Then for each vertex $v$ of $D$, $d^-(v)\leq 2$.
	\end{lemma}
	\begin{proof}
		Let $v$ be a vertex of $D$ with $d^-(v)>1$, then $f^-(v)$ is not domination-forced, so it is location-forced. It means that there is a vertex $u\in V(D)$
		such that $N^-(v)\setminus N^-(u)= \{f^-(v)\}$. Now, for contrary suppose that $d^-(v)\geq 3$. If $u$ is an in-neighbour of $v$, then $u$ and $v$ should have at least one common in-neighbour other than $v$, hence $T$ contains a cycle of length $3$, which is a contradiction. Otherwise,	
		vertices $v$ and $u$ have at least have two common in-neighbours, so in this case $T$ contains a cycle of length $4$, which is again a contradiction. Hence  $d^-(v)\leq 2$, as desired.
	\end{proof}
	
	Recall that by Theorem~\ref{thmmain}, the forcing arcs of an extremal digraph $D$ induce a disjoint union of directed cycles that spans the entire vertex set of $D$. If $D$ is a di-tree, 
	then these cycles are either loops or directed 2-cycles. In particular, if a vertex is loop-free, it is incident with a directed 2-cycle.

	\begin{lemma}\label{lemoutdegreeL-F}
		Let $D\in \mathcal{T}_n$ and $T$ be the underlying tree of $D$. Then for each location-forced vertex $v$ of $D$, $d^+(v)=1$.
	\end{lemma}
	
	\begin{proof}
		Let $v$ be a location-forced vertex of $D$. For contrary suppose that there exits a vertex  $x\neq f^+(v)$
		such that $vx \in A(D)$. By Lemma~\ref{lemtd-<2}, we have $N^-(x)=\{f^-(x),v\}$. 
		Hence, $f^-(x)$ is not domination-forced and so it is location-forced. Therefore, there exists a vertex $y$
		such that $N^-(x)\setminus N^-(y)= \{f^-(x)\}$. We conclude that $N^-(y)=\{v\}$, this means that $v$ is domination-forced which contradicts Proposition \ref{propblueblue}.
	\end{proof}

	We next prove two structural lemmas.
	
	\begin{lemma}\label{lem-redloop-leaf}
		Let $D\in \mathcal{T}_n$ and $T$ be the underlying tree of $D$.
		Suppose that $a$ is a leaf in $T$ with a forcing loop attached to $a$. Then there is no cycle of length~2 in $D$ which contains $a$.
	\end{lemma}
	\begin{proof}
		Towards a contradiction, suppose that there is a cycle of length~2 containing the arcs $ab$ and $ba$ (by Theorem\ref{thmmain}, we know that none of these two arcs are forcing arcs). Since $d^-(a)=2$, we conclude that $f^-(a)=a$ is not domination-forced, and hence by Proposition~\ref{prop:all-forced}, it is location-forced.
		Since $aa$ is a forcing loop, we conclude that there exists a vertex $c$ such that $N^-(a)\setminus N^-(c)=\{a\}$.  Using $N^-(a)=\{a,b\}$, we have $N^-(c)=\{b\}$
		(note that $b\neq c$). Therefore, $b$ is domination-forced, and since all forcing cycles are of length at most~2, we conclude that $bc$ and $cb$ are both forcing arcs. Since $d^-(b)\geq 2 $ (in fact by Lemma~\ref{lemtd-<2}, $d^-(b)= 2$), $c$ cannot be domination-forced, and by Proposition~\ref{prop:all-forced}, it is location-forced.
		Since $cb$ is a forcing arc, we conclude that there is a vertex $f$
		such that $N^-(b)\setminus N^-(f)=\{c\}$, which means that $N^-(f)=\{a\}$.  The latter means that $af$ is also a forcing arc. We note that since $d^-(a)=2$ and $d^-(f)=1$, $f\neq a$. Therefore, $a$ is contained  in two different forcing-cycles, which contradicts Theorem~\ref{thmmain}. Thus, the proof is complete.
	\end{proof}
	
	\begin{lemma}\label{lemtreecond}
		Let $D\in \mathcal{T}_n$, $T$ be the underlying tree of $D$ and $v$ be a leaf of $T$. Then
		at least one of the following conditions hold: 
		\begin{enumerate}
			\item  $T$ contains a pendant path of length~2 whose leaf is contained in a forcing $2$-cycle.
			\item  $T$ contains a leaf which is included in a forcing loop.
		\end{enumerate}
		
	\end{lemma}
	\begin{proof}
		We recall that by Theorem~\ref{thmmain}, every leaf belongs to a unique forcing cycle in $D$ (of length at most~2). If $T$ contains a pendant path of length~2, then its leaf is either contained in a forcing cycle of length $2$ or of length one and hence $D$ satisfies at least one of the mentioned conditions. Otherwise, $T$ should not have any pendant path of length~2. Therefore, $T$ contains two leaves adjacent to the same vertex. Now, using  Theorem~\ref{thmmain}, at least one of these two leaves must have a forcing loop attached, which concludes the proof.
	\end{proof}

	\subsection{The case of a leaf with a forcing loop attached}
	
	Next, we give a recursive construction for digraphs
	$D\in \mathcal{T}_n$, which contain a forcing loop on a leaf. To this aim, we will use digraphs $D'\in \mathcal{T}_{n-1}$.
	
	\begin{lemma}\label{lemrecursive1}
		Let $n>2$ be an integer, $D\in \mathcal{T}_n$ and $T$ be the underlying graph of $D$. Suppose that $a$ is a leaf in $T$ with a forcing loop attached and $b$ is the unique neighbour of $a$ in $T$.
		Letting $D'=D\setminus \{a\}$, then $D'\in \mathcal{T}_{n-1}$. Moreover, if $ba\in A(D)$, then $b$ is domination-forced in $D$ and also in $D'$. If
		$ab\in A(D)$, then $f^-_{D'}(b)$ is domination-forced in $D'$, $d^+_{D}(f_D^-(b))=1$ and $b$ has no loop attached.

	\end{lemma}		
	\begin{proof}
		By Proposition~\ref{prop:all-forced}, $a$ is either location-forced or domination-forced in $D$. First suppose that $a$ is location-forced. Hence, $a$ is not the unique neighbour of itself. Using Lemma \ref{lemoutdegreeL-F}, we conclude that $d^+(a)=1$, hence $ab\not\in A(D)$. On the other hand, using Lemma~\ref{lemtd-<2}, we conclude that $N^-(a)=\{a,b\}$.  As $a$ is location-forced, there is a vertex $x$
		such that $N_D^-(x)=N_D^-(a)\setminus \{a\}=\{b\}$, 
		and hence $b=f^-(x)$, or equivalently $b$ is domination-forced (in both $D$ and $D'$).

		Now, suppose that $a$ is domination-forced, which implies that $ba\not \in A(D)$, so $ab\in A(D)$. We show that there is no loop at $b$ in $D$.
		Since otherwise by Lemma~\ref{lemtd-<2}, $d^-_D(b)= 2$ and $N^-_D(b)=\{a,b\}$, hence, $N^-_D(b)\setminus N^-_D(a)=\{b\}$ and $bb$ is a forcing arc in $D$. Therefore, $f_D^-(b)=b$, hence $b$ is location-forced in $D$. Since $D$ is connected and has more than two vertices, we conclude that there is a vertex $c\neq b$ in $D'$ such that $bc\in A(D)$. This contradicts Lemma \ref{lemoutdegreeL-F}. 
		Therefore, there is no loop at $b$ and $f_D^-(b)\neq b$. By Lemma~\ref{lemtd-<2}, we have $d^-_D(b)=2$, hence $d^-_{D'}(b)=1$ and so, $f_{D'}^-(b)$ is domination-forced in $D'$. 	
		Now, since  $f^-_{D}(b)$ is location-forced, by Lemma \ref{lemoutdegreeL-F} we conclude that $d^+(f^-_{D}(b))=1$, as desired.

		To complete the proof of the lemma, we must show that if $a$ is domination-forced, then $D'\in \mathcal{T}_{n-1}$. By deleting the vertex $a$ and its incident arcs, for every vertex $v\neq b$ of $D'$, we have $N^-_{D'}(v)=N^-_D(v)$. On the other hand, by Lemma~\ref{lemtd-<2}, $d^-_D(b)=2$ and so, $d^-_{D'}(b)=1$, which shows that $x=f_D^-(b)$ is domination-forced in $D'$. Hence, all domination-forced (resp. location-forced) vertices of $V(D)\setminus \{x\}$ remain domination-forced (resp. location-forced) in $D'$, and $x$ is domination-forced. Therefore, all vertices are forced and $D'\in \mathcal{T}_{n-1}$, as desired.
	\end{proof}
	
	We now show the converse of Lemma~\ref{lemrecursive1}.
	
	\begin{lemma}\label{lemrecursive1-construct}
		Let $n>1$ be an integer, $D'\in \mathcal{T}_{n-1}$ and $b\in V(D')$. Suppose that $D$ is a digraph with $V(D)=V(D')\cup\{a\}$ and the arc set of $D$ is defined using one of the following rules.
		\begin{itemize}
			\item[i.]  If $b$ is domination-forced in $D'$, then $A(D)=A(D')\cup \{ba, aa\}$.
			\item[ii.] If $bb\not\in A(D)$, $d^+_{D'}(f^-(b))=1$ and
			$d^-_{D'}(b)=1$ in $D'$,  then
			$A(D)=A(D')\cup \{ab, aa\}$.
		\end{itemize}
		Then, $D\in \mathcal{T}_n$.
	\end{lemma}
	
	\begin{proof}
		\begin{itemize}
			\item [i.]
			Since $b$ is domination-forced, $N_D^-(a)\setminus N_D^-(f^-(b))=\{a\}$, therefore, $a$ is location-forced in $D$. On the other hand, if a vertex is domination-forced (resp. location-forced) in $D'$, then it is
			domination-forced (resp. location-forced) in $D$. Hence, $D\in \mathcal{T}_n$, as desired.
			\item [ii.] In this case, since  $N_D^-(a)=\{a\}$, we conclude that $a$ is domination-forced. Since $N_D^-(b)\setminus N_D^-(a)=\{f_{D'}^-(b)\}$, $f^-_{D'}(b)$ is location-forced in $D$. Moreover, it is easy to see that all domination-forced vertices of $D'$ except $f^-_{D'}(b)$, remain domination-forced in $D$, and since  $d^+_{D'}(f^-(b))=1$, all location-forced vertices in $D'$ remain location-forced in $D$. Hence, $D\in \mathcal{T}_n$.\qedhere
		\end{itemize}
	\end{proof}

	\subsection{The case of a pendant path of length~2 whose leaf is contained in a forcing 2-cycle}
	
	In the following lemma, we give a recursive construction for digraphs
	$D\in \mathcal{T}_n$ with underlying tree $T$,
	in which $T$ contains a pendant path of length~2 whose leaf is contained in a forcing $2$-cycle. In this recursive construction, we will use digraphs $D'\in \mathcal{T}_{n-2}$.
	
	\begin{lemma}\label{lempendant}
		Let $n\geq 3$ be an integer, $D\in \mathcal{T}_n$ and $T$ be the underlying tree of $D$. Let $P=c b a$ be a pendant path of length~2 in $T$. Assume that $d_T(a)=1$, $D'=D\setminus \{a,b\}$ and vertices $a,b$ are contained in a common forcing $2$-cycle. Then, $D'\in\mathcal{T}_{n-2}$. Moreover, the  following conditions hold:
		\begin{itemize}
			\item[i.] $aa\not \in A(D)$;
			\item[ii.] If $bb\in A(D)$, then $cb\not\in A(D)$. Moreover, $f^-(c)$ is location-forced in $D$ and domination-forced in $D'$. Furthermore, if $c=	f^-(c)$, then the only possibility for $D$ is  the digraph shown in Figure \ref{fig:c=f}(a).
			\item[iii.] If $bb\not \in A(D)$ and $f^-(c)$ is domination-forced in $D$, then $N^-_D(c)=\{f^-(c)\}$ and $f^-(c)\neq b$. Hence, $bc \not \in A(D)$ and since $D$ is connected, $cb\in A(D)$. 
			\item[iv.] If $bb\not \in A(D)$ and $c$ is domination-forced in $D$ and $f^-(c)$ is location-forced, then $d^+_D(f^-(c))=1$.
			
			\item[v.] If $bb\not \in A(D)$ and $c$ and $f^-(c)$ are both location-forced in $D$, then $cb\not\in A(D)$ and $bc\in A(D)$. Moreover, if $c=f^-(c)$, then the only possibility for $D$ is the digraph shown in Figure~\ref{fig:c=f}(b).
		\end{itemize}

	\end{lemma}

	\begin{figure}[!htpb]
		\centering
		\begin{tikzpicture}
			[
			-{Latex[scale=1.1]},
			auto,
			node distance = 2cm, 
			thick 
			]
			
			\tikzset{every state}=[
			draw = black,
			thick,
			fill = white,
			minimum size = 1mm
			]
			\begin{scope}[xshift=10.5cm,yshift=-5cm]
				
				\node[shape=circle,draw=black] (1) at (0,0) {$a$};
				\node[shape=circle,draw=black] (2) at (2.5,0) {$b$};
				\node[shape=circle,draw=black] (3) at (5,0) {$c$};
				
				\path[draw,very thick,dashed] (1) edge[bend left] (2);
				\path[draw,very thick,dashed] (2) edge[bend left] (1);
				\path [draw,very thick, color=black](2) edge node[left] {} (3);
				\path[draw,very thick,dashed] (3) edge[out=50,in=100,loop, min distance=10mm] node {} (3);
				\path[draw,very thick] (2) edge[out=50,in=100,loop, min distance=10mm] node {} (2);
				\node (title) at (2.5,-3/2) {(a) };
			\end{scope}

			\begin{scope}[xshift=10.5cm,yshift=-5cm]

				\node[shape=circle,draw=black] (1) at (8,0) {$a$};
				\node[shape=circle,draw=black] (2) at (10.5,0) {$b$};
				\node[shape=circle,draw=black] (3) at (13,0) {$c$};
				
				\path[draw,very thick,dashed] (1) edge[bend left] (2);
				\path[draw,very thick,dashed] (2) edge[bend left] (1);
				\path [draw,very thick, color=black](2) edge node[left] {} (3);
				\path[draw,very thick,dashed] (3) edge[out=50,in=100,loop, min distance=10mm] node {} (3);
				\node (title) at (10.5,-3/2) {(b) };
			\end{scope}

		\end{tikzpicture}
		\caption{Two extremal di-trees of order~3. Forcing arcs are dashed.}\label{fig:c=f}
	\end{figure}
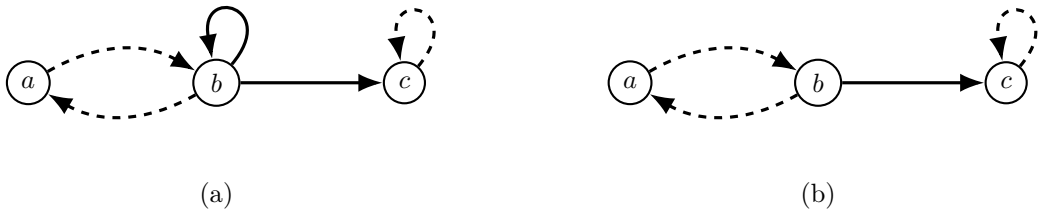

	\begin{proof}
		
		In the following, we prove that $D$ satisfies the claimed conditions.
		\begin{itemize}
			\item[i.] For a contradiction, suppose that $aa\in A(D)$. By Lemma~\ref{lemtd-<2}, we conclude that $d^-(a)=2$, and thus $b=f^-(a)$ is location-forced; then there is a vertex $x$ such that $N^-(a)\setminus N^-(x)=\{b\}$, which shows that $N^-(x)=\{a\}$. Since $a$ is a leaf in $T$, we conclude that $x=b$ and $cb\not \in A(D)$. Since $b$ is location-forced, by Lemma~\ref{lemoutdegreeL-F}, $d^+_D(b)=1$ and $bc \not \in A(D)$. Since the underlying graph of $D$ is connected we conclude that $a$ and $b$ are the only vertices of $D$, which contradicts the assumption that $n\geq 3$.
			\item[ii.] Since $ab\in A(D)$, using Lemma~\ref{lemtd-<2}, we conclude that $cb\not\in A(D)$ and so $bc\in A(D)$. Since $b\neq f^-(c)$, by Lemma \ref{lemtd-<2}, $d^-(c)=2$. Therefore, we conclude that $f^-(c)$ is location-forced in $D$ (and domination-forced in $D'$) and by Lemma~\ref{lemoutdegreeL-F}, $d^+_D(f^-(c))=1$.
			
			Moreover, if $c=f^-(c)$, then $d^+_D(c)=1$ and by Lemma~\ref{lemtd-<2}, $N^-_D(c)=\{c,b\}$. Hence, the vertex $c$ does not have any in-neighbour, other than $b$ and $c$ nor any out-neighbour, other than $c$, in $D$. Therefore, the only possibility for $D$ is the digraph shown in Figure~\ref{fig:c=f}(a).
			
			\item[iii.]  Suppose $bb\not \in A(D)$ and $f^-(c)$ is domination-forced in $D$. Then, $N^-_D(c)=\{f^-(c)\}$ and $f^-(c)\neq b$. Hence, $bc \not \in A(D)$.
			
			\item[iv.] We have $d^+(f^-(c))=1$ by Lemma~\ref{lemoutdegreeL-F}.
			
			\item[v.] By contradiction, suppose that $cb\in A(D)$. Since $b\neq f^+(c)$, we have $d^+(c)\geq 2$ which contradicts Lemma~\ref{lemoutdegreeL-F}.
			Thus, $cb\not\in A(D)$ and since $D$ is connected, $bc\in A(D)$.
			
			Now suppose that $c=f^-(c)$, then by Lemma~\ref{lemoutdegreeL-F}, $d^+_D(c)=1$ and by Lemma~\ref{lemtd-<2}, $N^-_D(c)=\{c,b\}$. Hence the vertex $c$ does not have any in-neighbour other than $b$ and $c$ or out-neighbour other than $c$, and the only possibility for $D$ is the digraph shown in Figure~\ref{fig:c=f}(b).\qedhere
		\end{itemize}	
	\end{proof}
	
	We now show the converse of Lemma~\ref{lempendant}.
	
	\begin{lemma}\label{lempendant-construct}
		Let $D'\in \mathcal{T}_{n-2}$ and $c$ be an arbitrary vertex of $D'$. Suppose that $D$ is a digraph with $V(D)=V(D')\cup\{a,b\}$ and the arc set of $D$ is defined using one of the following rules:
		\begin{itemize}
			\item[i.] If $c$ and $f^-(c)$ are both domination-forced in $D'$ and $d_{D'}^+(f^-(c))=1$, then $A(D)=A(D')\cup \{ ab, ba\}\cup A$, where $A\in \{\{bb, bc\}, \{cb,bc\}, \{bc\}, \{cb\}\}$.

			\item[ii.] If $c$ is location-forced in $D'$, $f^-(c)$ is domination-forced in $D'$ and $d_{D'}^+(f^-(c))=1$, then $A(D)=A(D')\cup \{bb, bc, ab, ba\}$ or $A(D)=A(D')\cup \{bc, ab, ba\}$.
			\item[iii.] If $c$  and $f^-(c)$ are both domination-forced in $D'$, and $d_{D'}^+(f^-(c))>1$ then $A(D)=A(D')\cup \{cb, ab, ba\}$.
			\item[iv.] If $c$ is domination-forced in $D'$  and $f^-(c)$ is location-forced in $D'$,  then $A(D)=A(D')\cup \{cb, ab, ba\}$.
		\end{itemize}
		Then $D\in \mathcal{T}_n$.
	\end{lemma}
	
	\begin{proof} First we note that if
		$d_{D'}^+(f^-(c))=1$ (cases i and ii), then $N_{D'}^+(f^-(c))=\{c\}$. So in these cases there is no vertex $y$ such that $N^-(y)\setminus N^-(c)=\{f^-(y)\}$. Hence, if the new digraph $D$ is constructed by adding some new in-neighbours to $c$, this does not affect the forcing vertices of $D'$, other than $f^-(c)$. Thus, to prove that $D$ is extremal in cases i and ii, it suffices to show that by adding the set of new arcs, each vertex from the set $\{a,b,c,f^-(c)\}$ is a forced vertex in $D$.
		
		Moreover, in cases iii and iv, we do not add any in-neighbours to $c$, so in these cases as well it suffices to show that after adding the new arcs, each vertex from $\{a,b,c,f^-(c)\}$ is a forced vertex in $D$.
	
		\begin{itemize}
			\item [i.] 
			As the vertices $c$ and $f^-(c)$ are both domination-forced in $D'$, using Definitions \ref{def-f-f+} and \ref{def-ditree}, we conclude that $N_{D'}^-(f^-(c))=\{c\}$
			and $N_{D'}^-(c)=\{f^-(c)\}$. We claim that if $D'$ has more than one vertex, then $f^-(c)\neq c$. By contradiction, suppose that  $f^-(c)= c$, this means that there is a forcing loop at $c$. Since $d^-(c)=1$ and $d^+(f^-(c))=d^+(c)=1$ and using the fact that $D'$ is connected, we conclude that $V(D')=\{c\}$, which is a contradiction. Hence, the claim is true and $f^-(c)\neq c$. Now, we prove that in this case, $c$ remains domination-forced in $D$. To prove this, we note that  $N_{D'}^-(f^-(c))=\{c\}$ and $c\neq f^-(c)$. Therefore  $N_{D}^-(f^-(c))=\{c\}$, which shows that $c$ is domination-forced in $D$.
			
			Therefore, if $A(D)=A(D')\cup \{bb, bc, ab, ba\}$, then $N_{D}^-(a)=\{b\}$, $N_{D}^-(b)=\{a,b\}$ and $N_{D}^-(c)=\{f^-(c),b\}$. Hence, $b$ is domination-forced in $D$, $a$ and $f^-(c)$ are both location-forced in $D$.
			
			If $A(D)=A(D')\cup \{cb, bc, ab, ba\}$, then $N_{D}^-(a)=\{b\}$, $N_{D}^-(b)=\{a,c\}$ and $N_{D}^-(c)=\{f^-(c),b\}$. Hence, $b$ is domination-forced in $D$, $f^-(c)$ and $a$ are both location-forced in $D$ (the latter because there is a vertex in $D$ only dominated by $c$).
			
			If $A(D)=A(D')\cup \{ bc, ab, ba\}$, $N_{D}^-(a)=\{b\}$, then $N_{D}^-(b)=\{a\}$ and $N_{D}^-(c)=\{f^-(c),b\}$. Hence, $b$ and $a$ are both domination-forced in $D$ and $f^-(c)$ is location-forced in $D$.
			
			Finally, If $A(D)=A(D')\cup \{cb, ab, ba\}$, then $N_{D}^-(a)=\{b\}$, $N_{D}^-(b)=\{a,c\}$ and $N_{D}^-(c)=\{f^-(c)\}$. Hence, $b$ and $f^-(c)$ are both domination-forced in $D$ and $a$ is location-forced in $D$ (because there is a vertex in $D$ only dominated by $c$).
			
			In all cases, $c$ remains domination-forced.
			Therefore, we conclude that each vertex of $D$ is either domination-forced or location-forced which implies that $D\in \mathcal{T}_n$, as desired.
			
			\item [ii.] Since $f^-(c)$ is domination-forced in $D'$, $N_{D'}^-(c)=\{f^-(c)\}$.
			First suppose that $A(D)=A(D')\cup \{bb, bc, ab, ba\}$. Since $N_D^-(b)=\{a\}\cup N_D^-(a)$ and $N_D^-(a)=\{b\}$, we conclude that $a$ is location-forced and $b$ is domination-forced in $D$. Since $N_D^-(c)=\{f^-(c)\}\cup N_D^-(a)$, $f^-(c)$ is location-forced in $D$ and one can see that $c$ remains location-forced in $D$. Therefore, in this case $D\in \mathcal{T}_n$.
			
			Now, suppose that $A(D)=A(D')\cup \{bc, ab, ba\}$. Since $N_D^-(b)=\{a\}$ and $N_D^-(a)=\{b\}$, we have that $a$ and $b$ are both domination-forced in $D$. On the other hand,  $N_D^-(c)=\{f^-(c)\}\cup N_D^-(a)$, so $f^-(c)$ is location-forced in $D$ and again one can see that $c$ remains location-forced in $D$.

			Hence, $D\in \mathcal{T}_n$.
			\item [iii.] Since $c$ and $f^-(c)$ are both domination-forced in $D'$, using Definitions \ref{def-f-f+} and \ref{def-ditree}, we have $N_{D'}^-(f^-(c))=\{c\}$ and $N_{D'}^-(c)=\{f^-(c)\}$.

			Considering  $N_{D}^-(a)=\{b\}$, $N_{D}^-(b)=\{a,c\}$, $N_{D}^-(f^-(c))=\{c\}$ and $N_{D}^-(c)=\{f^-(c)\}$, it is easy to see that $b$, $c$ and $f^-(c)$ are all domination-forced in $D$ and $a$ is location-forced. Hence, $D\in \mathcal{T}_n$ as desired.
			\item [iv.]
			It is easy to see that $b$ and $c$ are domination-forced  and $f^-(c)$ remains location-forced in $D$. Since $c$ is domination-forced in $D'$, $N_{D}^-(f^-(c))=N_{D'}^-(f^-(c))=\{c\}$, hence $N_{D}^-(b)=\{a\}\cup N_{D}^-(f^-(c))$ and $a$ is location-forced in $D$. Therefore, $D\in \mathcal{T}_n$.\qedhere
		\end{itemize}
	\end{proof}
	
	\subsection{The characterization}

	As a conclusion of this section, we give our characterization theorem which shows how digraphs in $\mathcal{T}_n$ can be constructed recursively, using extremal digraphs of smaller order.
	
	\begin{definition}
		Let $\mathcal{C}^1(\mathcal{T}_n)$ be the set of all digraphs $D\in \mathcal{T}_{n+1}$ which are constructed from a digraph $D'\in \mathcal{T}_n$ using one of the rules given in Lemma~\ref{lemrecursive1-construct}, and $\mathcal{C}^2(\mathcal{T}_n)$ be the set of all digraphs $D\in \mathcal{T}_{n+2}$ which are constructed from a digraph $D'\in \mathcal{T}_{n}$ using one of the rules given in Lemma~\ref{lempendant-construct}.
	\end{definition}

	\begin{theorem}\label{thm-main}
		Let $n$ be a positive integer. If $n\leq 2$, then all extremal digraphs of $\mathcal{T}_{1}\cup \mathcal{T}_{2}$ are shown in Figure~\ref{fig:order1,2}. If $n>2$, then, we have $\mathcal{T}_n=\mathcal{C}^1(\mathcal{T}_{n-1})\cup \mathcal{C}^2(\mathcal{T}_{n-2})$.
	\end{theorem}
	\begin{proof}
		For $n\leq 2$, all connected locatable digraphs of these orders are those of Figure~\ref{fig:order1,2} and they are all extremal. Thus, assume next that $n>2$.
		
		By Lemma~\ref{lemrecursive1-construct}, we have $\mathcal{C}^1(\mathcal{T}_{n-1})\subseteq \mathcal{T}_n$ and by Lemma~\ref{lempendant-construct}, we have $\mathcal{C}^2(\mathcal{T}_{n-2})\subseteq \mathcal{T}_n$.
		
		Conversely, to see that $\mathcal{T}_n\subseteq\mathcal{C}^1(\mathcal{T}_{n-1})\cup \mathcal{C}^2(\mathcal{T}_{n-2})$, assume that we have a digraph $D$ in $\mathcal{T}_n$ whose underlying tree is $T$. 
		If $D$ contains a forcing loop at a leaf of $T$, then Lemma~\ref{lemrecursive1} shows that $D$ can be constructed from a digraph of $\mathcal{T}_{n-1}$ by one of the rules in Lemma~\ref{lemrecursive1-construct} and thus $D\in \mathcal{C}^1(\mathcal{T}_{n-1})$. Otherwise, using  Lemma~\ref{lemtreecond}, $T$ contains a pendant path of length~2 whose leaf is contained in a forcing 2-cycle in $D$. Hence, by Lemma~\ref{lempendant}, $D$ can be constructed from a digraph of $\mathcal{T}_{n-2}$ by one of the rules in Lemma~\ref{lempendant-construct} and thus $D\in \mathcal{C}^2(\mathcal{T}_{n-2})$.
	\end{proof}
	
			We depict in Figure~\ref{fig:smalltrees} all extremal digraphs of order at most~4, that were constructed using Theorem~\ref{thm-main}.

			\begin{figure}[!htpb]
				\centering
				\begin{tikzpicture}
					\node[small node](x00) at (0,5)    {};
					
					\path[thick,densely dashed,-{Latex[scale=1.1]}] (x00) edge[out=50,in=100,loop, min distance=10mm] node {} (x00);
					
					\node[small node](x01) at (1.5,5)    {};
					\node[small node](y01) at (1.5,6.5)    {};
					
					\path[thick,densely dashed,-{Latex[scale=1.1]}] (x01) edge[bend left] (y01)	(y01) edge[bend left] (x01);

					\node[small node](x02) at (3,5)    {};
					\node[small node](y02) at (3,6.5)    {};
					
					\path[thick,densely dashed,-{Latex[scale=1.1]}] (x02) edge[bend left] (y02)	(y02) edge[bend left] (x02);
					\path[thick,-{Latex[scale=1.1]}] (x02) edge[out=30,in=330,loop, min distance=8mm] node {} (x02);

					\node[small node](x03) at (4.5,5)    {};
					\node[small node](y03) at (4.5,6.5)    {};
					
					\path[thick,-{Latex[scale=1.1]}] (x03) edge node {} (y03);
					\path[thick,densely dashed,-{Latex[scale=1.1]}] (x03) edge[out=330,in=30,loop, min distance=8mm] node {} (x03);
					\path[thick,densely dashed,-{Latex[scale=1.1]}] (y03) edge[out=330,in=30,loop, min distance=8mm] node {} (y03);

					\node[small node](x1) at (6,5)    {};
					\node[small node](y1) at (6,6.5)    {};
					\node[small node](z1) at (6,8)    {};
					
					\path[thick,densely dashed,-{Latex[scale=1.1]}] (x1) edge[bend left] (y1)	(y1) edge[bend left] (x1);
					\path[thick,-{Latex[scale=1.1]}] (y1) edge node {} (z1);
					\path[thick,densely dashed,-{Latex[scale=1.1]}] (z1) edge[out=330,in=30,loop, min distance=8mm] node {} (z1);

					\node[small node](x2) at (7.5,5)    {};
					\node[small node](y2) at (7.5,6.5)    {};
					\node[small node](z2) at (7.5,8)    {};
					
					\path[thick,densely dashed,-{Latex[scale=1.1]}] (x2) edge[bend left] (y2)	(y2) edge[bend left] (x2);
					\path[thick,-{Latex[scale=1.1]}] (y2) edge node {} (z2);
					\path[thick,-{Latex[scale=1.1]}] (y2) edge[out=330,in=30,loop, min distance=8mm] node {} (y2);
					\path[thick,densely dashed,-{Latex[scale=1.1]}] (z2) edge[out=330,in=30,loop, min distance=8mm] node {} (z2);

					\node[small node](x3) at (9,5)    {};
					\node[small node](y3) at (9,6.5)    {};
					\node[small node](z3) at (9,8)    {};
					
					\path[thick,densely dashed,-{Latex[scale=1.1]}] (x3) edge[bend left] (y3)	(y3) edge[bend left] (x3);
					\path[thick,-{Latex[scale=1.1]}] (z3) edge node {} (y3);
					\path[thick,densely dashed,-{Latex[scale=1.1]}] (z3) edge[out=330,in=30,loop, min distance=8mm] node {} (z3);

					\node[small node](x4) at (10.5,5)    {};
					\node[small node](y4) at (10.5,6.5)    {};
					\node[small node](z4) at (10.5,8)    {};
					
					\path[thick,densely dashed,-{Latex[scale=1.1]}] (x4) edge[out=330,in=30,loop, min distance=8mm] node {} (x4);
					\path[thick,-{Latex[scale=1.1]}] (y4) edge node {} (x4);
					\path[thick,densely dashed,-{Latex[scale=1.1]}] (y4) edge[out=330,in=30,loop, min distance=8mm] node {} (y4);
					\path[thick,-{Latex[scale=1.1]}] (y4) edge node {} (z4);
					\path[thick,densely dashed,-{Latex[scale=1.1]}] (z4) edge[out=330,in=30,loop, min distance=8mm] node {} (z4);

					\node[small node](a01) at (0,4)    {};
					\node[small node](b01) at (0,2.5)    {};
					\node[small node](c01) at (0,1)    {};
					\node[small node](d01) at (0,-0.5)    {};
					
					\path[thick,densely dashed,-{Latex[scale=1.1]}] (a01) edge[bend left] (b01)	(b01) edge[bend left] (a01);
					\path[thick,-{Latex[scale=1.1]}] (b01) edge node {} (c01);
					\path[thick,densely dashed,-{Latex[scale=1.1]}] (c01) edge[bend left] (d01)	(d01) edge[bend left] (c01);

					\node[small node](a02) at (1.6,4)    {};
					\node[small node](b02) at (1.6,2.5)    {};
					\node[small node](c02) at (1.6,1)    {};
					\node[small node](d02) at (1.6,-0.5)    {};
					
					\path[thick,densely dashed,-{Latex[scale=1.1]}] (a02) edge[bend left] (b02)	(b02) edge[bend left] (a02);
					\path[-{Latex[scale=1.1]}] (b02) edge[bend left]  (c02)  (c02) edge[bend left]  (b02);
					\path[thick,densely dashed,-{Latex[scale=1.1]}] (c02) edge[bend left] (d02)	(d02) edge[bend left] (c02);

					\node[small node](a03) at (3.2,4)    {};
					\node[small node](b03) at (3.2,2.5)    {};
					\node[small node](c03) at (3.2,1)    {};
					\node[small node](d03) at (3.2,-0.5)    {};
					
					\path[thick,densely dashed,-{Latex[scale=1.1]}] (a03) edge[bend left] (b03)	(b03) edge[bend left] (a03);
					\path[thick,-{Latex[scale=1.1]}] (b03) edge[out=30,in=330,loop, min distance=8mm] node {} (b03);
					\path[thick,-{Latex[scale=1.1]}] (b03) edge node {} (c03);
					\path[thick,densely dashed,-{Latex[scale=1.1]}] (c03) edge[bend left] (d03)	(d03) edge[bend left] (c03);

					\node[small node](a04) at (5,4)    {};
					\node[small node](b04) at (5,2.5)    {};
					\node[small node](c04) at (5,1)    {};
					\node[small node](d04) at (5,-0.5)    {};
					
					\path[thick,densely dashed,-{Latex[scale=1.1]}] (a04) edge[bend left] (b04)	(b04) edge[bend left] (a04);
					\path[thick,-{Latex[scale=1.1]}] (c04) edge node {} (b04);
					\path[thick,densely dashed,-{Latex[scale=1.1]}] (c04) edge[out=330,in=30,loop, min distance=8mm] node {} (c04);
					\path[thick,-{Latex[scale=1.1]}] (c04) edge node {} (d04);
					\path[thick,densely dashed,-{Latex[scale=1.1]}] (d04) edge[out=330,in=30,loop, min distance=8mm] node {} (d04);

					\node[small node](a4) at (6.8,4)    {};
					\node[small node](b4) at (6.8,2.5)    {};
					\node[small node](c4) at (6.8,1)    {};
					\node[small node](d4) at (6.8,-0.5)    {};
					
					\path[thick,-{Latex[scale=1.1]}] (b4) edge node {} (a4);
					\path[thick,densely dashed,-{Latex[scale=1.1]}] (b4) edge[bend left] (c4)	(c4) edge[bend left] (b4);
					\path[thick,densely dashed,-{Latex[scale=1.1]}] (a4) edge[out=330,in=30,loop, min distance=8mm] node {} (a4);
					\path[thick,densely dashed,-{Latex[scale=1.1]}] (d4) edge[out=330,in=30,loop, min distance=8mm] node {} (d4);
					\path[thick,-{Latex[scale=1.1]}] (c4) edge node {} (d4);

					\node[small node](a5) at (8.7,4)    {};
					\node[small node](b5) at (8.7,2.5)    {};
					\node[small node](c5) at (8.7,1)    {};
					\node[small node](d5) at (8.7,-0.5)    {};
					
					\path[thick,densely dashed,-{Latex[scale=1.1]}] (b5) edge[bend left] (a5)	(a5) edge[bend left] (b5);
					\path[thick,-{Latex[scale=1.1]}] (c5) edge node {} (b5);
					\path[thick,densely dashed,-{Latex[scale=1.1]}] (c5) edge[out=330,in=30,loop, min distance=8mm] node {} (c5);
					\path[thick,densely dashed,-{Latex[scale=1.1]}] (d5) edge[out=330,in=30,loop, min distance=8mm] node {} (d5);
					\path[thick,-{Latex[scale=1.1]}] (c5) edge node {} (d5);

					\node[small node](a7) at (10.5,4)    {};
					\node[small node](b7) at (10.5,2.5)    {};
					\node[small node](c7) at (10.5,1)    {};
					\node[small node](d7) at (10.5,-0.5)    {};	
					
					\path[thick,-{Latex[scale=1.1]}] (a7) edge node {} (b7);
					\path[thick,densely dashed,-{Latex[scale=1.1]}] (a7) edge[out=330,in=30,loop, min distance=8mm] node {} (a7);
					\path[thick,densely dashed,-{Latex[scale=1.1]}] (b7) edge[bend left] (c7)	(c7) edge[bend left] (b7);
					\path[thick,densely dashed,-{Latex[scale=1.1]}] (d7) edge[out=330,in=30,loop, min distance=8mm] node {} (d7);
					\path[thick,-{Latex[scale=1.1]}] (d7) edge node {} (c7);

					\node[small node](a1) at (0,-1.5)    {};
					\node[small node](b1) at (0,-3)    {};
					\node[small node](c1) at (0,-4.5)    {};
					\node[small node](d1) at (1.5,-3)    {};
					
					\path[thick,densely dashed,-{Latex[scale=1.1]}] (b1) edge[bend left] (a1)	(a1) edge[bend left] (b1);
					\path[thick,densely dashed,-{Latex[scale=1.1]}] (c1) edge[out=330,in=30,loop, min distance=8mm] node {} (c1);
					\path[thick,-{Latex[scale=1.1]}] (b1) edge node {} (c1);
					\path[thick,densely dashed,-{Latex[scale=1.1]}] (d1) edge[out=330,in=30,loop, min distance=8mm] node {} (d1);
					\path[thick,-{Latex[scale=1.1]}] (b1) edge node {} (d1);

					\node[small node](a2) at (3.5,-1.5)    {};
					\node[small node](b2) at (3.5,-3)    {};
					\node[small node](c2) at (3.5,-4.5)    {};
					\node[small node](d2) at (5,-3)    {};
					
					\path[thick,densely dashed,-{Latex[scale=1.1]}] (b2) edge[bend left] (a2)	(a2) edge[bend left] (b2);
					\path[thick,densely dashed,-{Latex[scale=1.1]}] (c2) edge[out=330,in=30,loop, min distance=8mm] node {} (c2);
					\path[thick,-{Latex[scale=1.1]}] (b2) edge[out=-20,in=-70,loop, min distance=8mm] node {} (b2);
					\path[thick,-{Latex[scale=1.1]}] (b2) edge node {} (c2);
					\path[thick,densely dashed,-{Latex[scale=1.1]}] (d2) edge[out=330,in=30,loop, min distance=8mm] node {} (d2);
					\path[thick,-{Latex[scale=1.1]}] (b2) edge node {} (d2);

					\node[small node](a3) at (6.5,-1.5)    {};
					\node[small node](b3) at (6.5,-3)    {};
					\node[small node](c3) at (6.5,-4.5)    {};
					\node[small node](d3) at (8,-3)    {};
					
					\path[thick,densely dashed,-{Latex[scale=1.1]}] (b3) edge[bend left] (a3)	(a3) edge[bend left] (b3);
					\path[thick,densely dashed,-{Latex[scale=1.1]}] (c3) edge[out=330,in=30,loop, min distance=8mm] node {} (c3);
					\path[thick,densely dashed,-{Latex[scale=1.1]}] (d3) edge[out=330,in=30,loop, min distance=8mm] node {} (d3);
					\path[thick,-{Latex[scale=1.1]}] (c3) edge node {} (b3);
					\path[thick,-{Latex[scale=1.1]}] (b3) edge node {} (d3);

					\node[small node](a6) at (9.5,-1.5)    {};
					\node[small node](b6) at (9.5,-3)    {};
					\node[small node](c6) at (9.5,-4.5)    {};
					\node[small node](d6) at (11,-3)    {};
					
					\path[thick,-{Latex[scale=1.1]}] (b6) edge node {} (a6);
					\path[thick,-{Latex[scale=1.1]}] (b6) edge node {} (c6);
					\path[thick,-{Latex[scale=1.1]}] (b6) edge node {} (d6);
					\path[thick,densely dashed,-{Latex[scale=1.1]}] (a6) edge[out=330,in=30,loop, min distance=8mm] node {} (a6);
					\path[thick,densely dashed,-{Latex[scale=1.1]}] (b6) edge[out=-20,in=-70,loop, min distance=8mm] node {} (b6);
					\path[thick,densely dashed,-{Latex[scale=1.1]}] (c6) edge[out=330,in=30,loop, min distance=8mm] node {} (c6);
					\path[thick,densely dashed,-{Latex[scale=1.1]}] (d6) edge[out=330,in=30,loop, min distance=8mm] node {} (d6);

				\end{tikzpicture}\centering
				\caption{All extremal di-trees with order at most~4. Forcing arcs are dashed.}\label{fig:smalltrees}
			\end{figure}
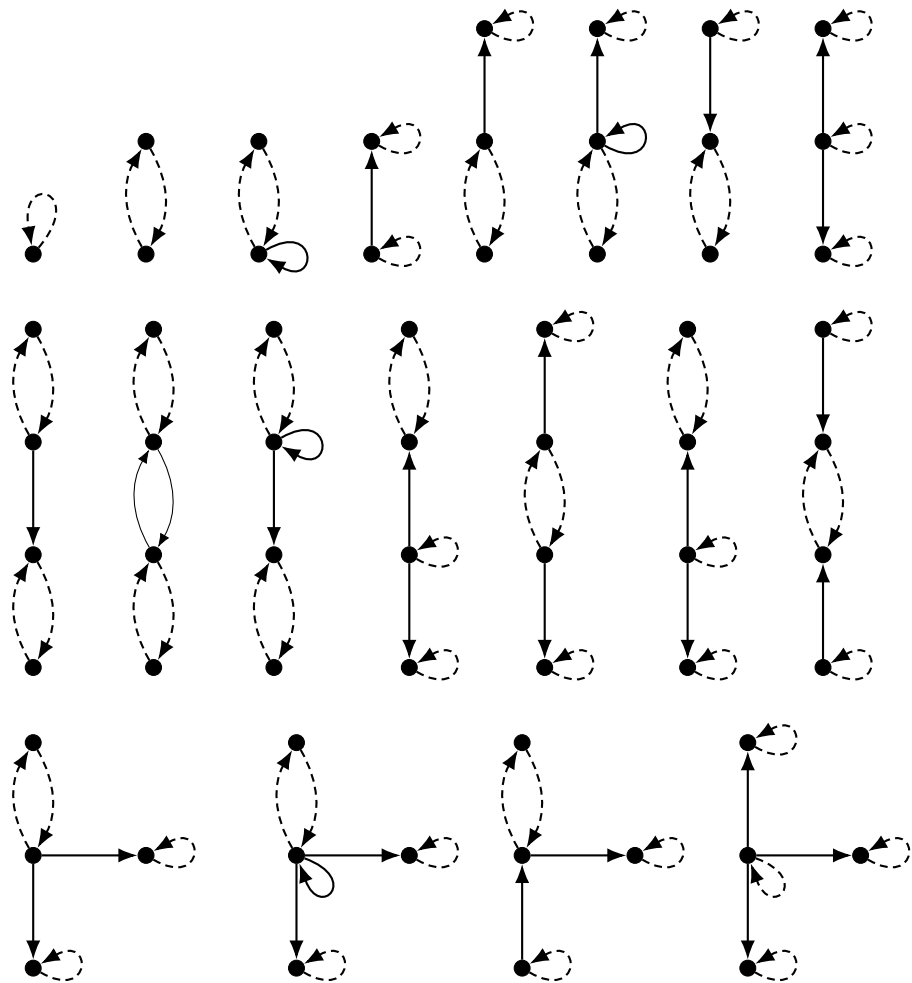

		\section{Conclusion}\label{conclu}
		
		By studying structural properties of extremal digraphs with respect to OLD sets, we have been able to give new proofs of several existing results about both digraphs and undirected graphs, for both identifying codes and OLD sets. Indeed, OLD sets of general digraphs generalize all these problems. Thus, we believe that our results shed new light on this type of extremal problems.
		
		We have also given a characterization of all such extremal digraphs, which, it appears, form a very rich class of digraphs. Even our recursive characterization for extremal di-trees, 
		although of course more restricted than the general case, shows that there are many such extremal trees.

	\end{document}